\documentclass[11pt]{article}
\usepackage{amsmath, amssymb, amsfonts, amsthm, amscd, vmargin}
\usepackage[latin1]{inputenc}
\usepackage[all]{xy}
\usepackage{graphicx}
\usepackage{enumitem}
\usepackage{tikz}
\usetikzlibrary{shapes}
\usepackage{url}

\setmarginsrb{2.4cm}{3.3cm}{2.4cm}{3.0cm}{0cm}{0cm}{0cm}{1.5cm}

\theoremstyle{plain}
\newtheorem{teo}{Theorem}[section]
\newtheorem{lem}[teo]{Lemma}
\newtheorem{prop}[teo]{Proposition}
\newtheorem{cor}[teo]{Corollary}

\theoremstyle{definition}
\newtheorem{defi}[teo]{Definition}

\newtheorem{rem}[teo]{Remark}
\newtheorem{rems}[teo]{Remarks}

\numberwithin{equation}{teo}

\newcommand{\Cbb}{{\mathbb C}}
\newcommand{\Qbb}{{\mathbb Q}}
\newcommand{\Zbb}{{\mathbb Z}}

\newcommand{\Rcal}{{\mathcal R}}
\newcommand{\Scal}{{\mathcal S}}
\newcommand{\Ical}{{\mathcal I}}

\begin{document}
\title{Klt singularities of horospherical pairs}
\author{Boris Pasquier}

\maketitle
\abstract{Let $X$ be a horospherical $G$-variety and let $D$ be an effective $\Qbb$-divisor of $X$ that is stable under the action of a Borel subgroup $B$ of $G$ and such that $D+K_X$ is $\Qbb$-Cartier. We prove, using Bott-Samelson resolutions, that the pair $(X,D)$ is klt if and only if $\lfloor D\rfloor=0$.}

\section{Introduction}

Let $X$ be a normal algebraic variety over $\Cbb$ and let $D$ be an effective $\Qbb$-divisor such that $D+K_X$ is $\Qbb$-Cartier. If the pair $(X,D)$ has klt singularities (see Definition~\ref{def:klt}) then $\lfloor D\rfloor=0$ (ie $D=\sum_{D_i\,\rm{irreducible}}a_iD_i$ with $a_i\in[0,1[$).
The inverse implication is false in general.
In \cite{AlexeevBrion}, V.~Alexeev and M.~Brion proved that, if $X$ is a spherical $G$-variety and $D$ be an effective $\Qbb$-divisor of $X$ such that $D+K_X$ is $\Qbb$-Cartier, $\lfloor D\rfloor=0$ and $D=D_G+D_B$ where $D_G$ is $G$-stable and $D_B$ is stable under the action of a Borel subgroup $B$ of $G$, then $(X,D_G+D_B')$ has klt singularities for general $D_B'$ in $|D_B|$.

Here, we prove that, if $X$ is a horospherical $G$-variety and $D$ be an effective $\Qbb$-divisor of $X$ such that $D+K_X$ is $\Qbb$-Cartier, $\lfloor D\rfloor=0$ and $D$ is stable under the action of a Borel subgroup $B$ of $G$, then the pair $(X,D)$ has klt singularities.

The strategy of the proof is the following.  In section~3, we recall the definitions and some properties of Bott-Samelson resolutions of any flag variety $G/P$. In particular, they are log resolutions and the klt singularity condition in the case of flag varieties becomes equivalent to some inequalities on the root systems of $G$ and $P\subset G$, which we prove in section~5. 
And in section~4, we deduce the horospherical case from the case of flag varieties, using that any horospherical variety admits a desingularization that is a toric fibration over a flag variety (ie a fibration over a flag variety whose fiber is a smooth toric variety).

\section{Notations and definitions}
In all the paper, varieties are algebraic varieties over $\Cbb$. 

We first recall the definition of klt singularities.

\begin{defi}\label{def:klt}
Let $X$ be a normal variety and let $D$ be an effective $\Qbb$-divisor such that $K_X+D$ is $\Qbb$-Cartier.
The pair $(X,D)$ is said to be klt (Kawamata log terminal) if for any resolution  $f:\,V\longrightarrow X$ of $X$ such that $K_V=f^*(K_X+D)+\sum_{i\in \mathcal{E}}a_iE_i$ where the $E_i$'s are distinct irreducible divisors, we have $a_i>-1$ for any $i\in \mathcal{E}$.
\end{defi}

\begin{rem}\label{rem:klt}
\begin{enumerate}
\item In fact, it is enough to check the above property for one log-resolution to say that a pair $(X,D)$ is klt. A log-resolution of $(X,D)$ is a resolution $f$ such that, the exceptional locus $\rm{Exc}(f)$ of $f$ is of pure codimension one and the divisor $f_*^{-1}(D)+\sum_{E\subset\rm{Exc}(f)}E$ has simple normal crossings (where $f_*^{-1}(D)$ is the strict transform of $D$ by $f$).
\item The condition "$a_i>-1$ for any $i\in \mathcal{E}$" can be replaced by: $\lfloor D\rfloor=0$ and for any $i\in \mathcal{E}$ such that $E_i$ is exceptional for $f$, $a_i>-1$.
\end{enumerate}
\end{rem}

In all the paper, $G$ denotes a connected reductive algebraic group over $\Cbb$.

Let $T$ be a maximal torus in $G$ and let $B$ be a Borel subgroup of $G$ containing $T$.
We denote by $\Rcal$ the root system of $(G,B,T)$, by $\Rcal^+$ the set of positive roots and by $\Scal$ the set of simple roots. For any simple root $\alpha\in\Scal$ we denote by $s_\alpha$ the corresponding simple reflection of the Weyl group $W=N_G(T)/T$. By abuse of notation, for any $w$ in $W$, we still denote by $w$ one of its representative in $G$. We denote by $w_0$ the longest element of $W$.

Let $P$ be a parabolic subgroup of $G$ that contains $B$. Denote by $\Ical$ the set of simple roots of $P$ (in particular, if $P=B$ we have $\Ical=\emptyset$ and, if $P=G$ we have $\Ical=\Scal$). Denote by $W_P$ the subgroup of $W$ generated by $\{s_\alpha\,\mid\,\alpha\in\Ical\}$. Also denote by $W^P$ the quotient $W/W_P$ and denote by $w_0^P$ the longest element of $W^P$.

The Bruhat decomposition of $G$ in $B\times B$-orbits gives the following decomposition of $G/P$:
$$G/P=\bigsqcup_{w\in W^P}BwP/P.$$ Moreover the dimension of a cell $BwP/P$ equals the length of $w$. In particular, the length of $w_0^P$ is the dimension of $G/P$ and irreducible $B$-stable divisors of $G/P$ are the closures of the cells $Bs_\alpha w_0^PP/P$ with $\alpha\in\Scal\backslash\Ical$. We denote them by $D_\alpha$.

A horospherical variety $X$ is a normal $G$-variety with an open $G$-orbit isomorphic to a torus fibration $G/H$ over a flag variety $G/P$ (ie $P/H$ is a torus). The irreducible divisors of such $X$ that are $B$-stable but not $G$-stable, are the closures in $X$ of the inverse images in $G/H$ of the Schubert divisors $D_\alpha$ of $G/P$ defined above. We still denote them by $D_\alpha$, with $\alpha\in\Scal\backslash\Ical$.

If $X$ and $Y$ are varieties such that a parabolic subgroup $P$ have a right action on $X$ and a left action on $Y$, we denote by $X\times^PY$ the quotient of the product $X\times Y$ by the following equivalences: $$\forall (x,y)\in X\times Y,\,\forall P\in P,\,\,(x,y)\sim(x\cdot p,p^{-1}\cdot y).$$
\section{Bott-Samelson desingularizations and klt pairs of flag varieties}

In that section, we prove the following result.

\begin{teo}\label{th:kltflag}
Let $D=\sum_{\alpha\in\Scal\backslash\Ical}d_\alpha D_\alpha$ be a $B$-stable $\Qbb$-divisor of $G/P$ such that $\forall\alpha\in\Scal\backslash\Ical,\,d_\alpha\in[0,1[$.

There exists a $B$-stable log-resolution $\phi:\,Z/P\longrightarrow G/P$ of $(G/P,D)$, where $Z$ is a variety with a right action of $P$ and a left action of $B$, such that the exceptional divisors of $\phi$ are the quotient by $P$ of irreducible divisors of $Z$, and such that $K_{Z/P}-\pi^*(K_{G/P}+D)=\sum_{i\in\mathcal{E}}a_iE_i$ where for any $i\in\mathcal{E}$, $a_i>-1$ and $E_i$ is an irreducible divisor of $f$.

In particular the pair $(G/P,D)$ is klt.

Moreover, for any $i\in\mathcal{E}$, $E_i$ is the quotient of an exceptional $B\times P$-stable divisor $F_i$ of $Z$ by $P$ (left action of $B$ and right action of $P$). 

\end{teo}

\begin{rems}
\begin{enumerate}[label=(\roman*)]
\item In general, $\sum_{\alpha\in\Scal\backslash\Ical}D_\alpha$ is not a simple normal crossing $\Qbb$-divisor of $G/P$. Then, it is not enough to know that $G/P$ is smooth to say that $(G/P,D)$ is klt, when $D\neq 0$.
\item Since $D$ is globally generated, then $(G/P,D')$ is klt for a general $D'$ in $|D|$ (consequence of \cite[Lemma~9.1.9]{Lazarsfeld}). We can generalized this remark to spherical pairs, see \cite[Theorem 5.3]{AlexeevBrion}.
\end{enumerate}
\end{rems}

To prove Theorem~\ref{th:kltflag}, we use a Bott-Sameslon resolution of $G/P$. Bott-Samelson resolution of Schubert varieties of $G/B$ have been introduced by M.~Demazure in \cite{DemazureBS}. Here, we use the easy (and well-known) generalization of his work to $G/P$. And we choose the equivalent definition of Bott-Samelson resolutions that is now used in almost all papers on the topic.

For any simple root $\alpha$, we denote by $P_\alpha$ the minimal parabolic subgroup containing $B$ such that $\alpha$ is a simple root of $P_\alpha$.

\begin{defi}
Let $s_{\alpha_1}s_{\alpha_2}\cdots s_{\alpha_N}$ be a reduced decomposition of $w_0^P$ with $\alpha_1\dots,\alpha_N$ in $\Scal$. We define the Bott-Samelson variety $BS$ to be the quotient of $P_{\alpha_1}\times P_{\alpha_2}\times\cdots\times P_{\alpha_N}$ by the right action of $B^N$ given by,
$$(p_1,p_2,\dots,p_N)\cdot(b_1,b_2,\dots,b_N)=(p_1b_1,b_1^{-1}p_2b_2,\dots,b_{N-1}^{-1}p_Nb_N).$$
\end{defi}

The map $\phi':\,BS\longrightarrow G/P$ that sends $(p_1,p_2,\dots,p_N)$ to $p_1p_2\cdots p_NP/P$ is well-defined and birational (it is an isomorphism from the quotient of $Bs_{\alpha_1}B\times Bs_{\alpha_2}B\times\cdots\times Bs_{\alpha_N}B$ by the right action of $B^N$ to $Bw_0^P/P$). (We can decompose this map by the usual map from $V$ to the Schubert variety $\overline{Bw_0^PB/B}$ of $G/B$ and the projection map from $G/B$ to $G/P$.)\\

Hence, to get $Z$ as in Theorem~\ref{th:kltflag}, we define $Z$ to be the quotient of $P_{\alpha_1}\times\cdots\times P_{\alpha_{N-1}}\times P_{{\alpha_N\cup\Ical}}$ by the right action of $B^{N-1}$ given by,
$$(p_1,\dots,p_N)\cdot(b_1,b_2,\dots,b_{N-1})=(p_1b_1,\dots,b_{N-1}^{-1}p_N).$$
Then, since $P_{{\alpha_N\cup\Ical}}/P=P_{\alpha_N}P/P\simeq P_{\alpha_N}/B$, the $B$-varieties $Z/P$ and $BS$ are isomorphic and $\phi:\,Z/P\longrightarrow G/P$ that sends $(p_1,\dots,p_N)$ to $p_1\cdots p_NP/P$ is well-defined and birational.

The lines bundles and divisors Bott-Samelson varieties are well-known, so that we can describe the lines bundles of $Z/P$, and the divisors of $Z/P$ and $Z$. 

\begin{prop}\label{prop:divBS}
For any $i\in\{1,\dots,N-1\}$, we define $F_i$ to be the $B\times P$-stable divisor of $Z$ defined by $p_i\in B$; and we define $F_N$ to be the $B\times P$-stable divisor of $Z$ defined by $p_N\in P$. 

Then, we can also define $E_i$ to be the $B$-stable divisor $F_i/P$ of $Z/P$. Moreover, the $B$-stable irreducible divisors of $Z/P$ are the $E_i$'s with $i\in\{1,\dots,N\}$, and the family $(E_i)_{i\in\{1,\dots,N\}}$ is a basis of the cone of effective divisors of $Z/P$. 
\end{prop}

First remark that the divisor $\sum_{i=1}^NE_i$ is clearly a simple normal crossing divisor. Also, since $G/P$ is smooth and by \cite[VI.1, Theorem~1.5]{Kollar}, we know that the exceptional locus of $\phi$ is of pure codimension one, so it is the union of the $E_i$'s contracted by $\phi$.\\

Now, let $\lambda$ be a character of $P$. It defines a line bundle $\mathcal{L}_{G/P}(\lambda)$ on $G/P$ (where $P$ acts on the fiber over $P/P$ by the character $\lambda$). And by pull-back by $\phi$, it defines a line bundle $\mathcal{L}_{Z/P}(\lambda)$ on $Z/P$. 

The total space of $\mathcal{L}_{Z/P}(\lambda)$ is the quotient of $P_{\alpha_1}\times \cdots\times P_{\alpha_{N-1}}\times P_{{\alpha_N\cup\Ical}}\times \Cbb$ by the right action of $B^{N-1}\times P$ given by,
$$(p_1,\dots,p_N,z)\cdot(b_1,\dots,b_{N-1},p)=(p_1b_1,b_1^{-1}p_2b_2,\dots,b_{N-1}^{-1}p_Np,\lambda(p)z).$$

By \cite[Section~2.5, Proposition~1]{DemazureBS} adapted to our notation and by induction on $N$, we have the following result.

\begin{prop}\label{prop:pullback}
Let $\lambda$ be a character of $P$. Then $\mathcal{L}_{G/P}(\lambda)$ is the line bundle associated to the $B$-stable divisor $D_\lambda:=\sum_{\alpha\in\Scal\backslash\Ical}\langle\lambda,\alpha^\vee\rangle D_\alpha$.

Moreover, $\phi^*(D_\lambda)=\sum_{i=1}^N\langle\lambda,\beta_i^\vee\rangle E_i$, where for any $i\in\{1,\dots,N\}$, $\beta_i=s_{\alpha_1}\cdots s_{\alpha_i-1}(\alpha_i)$.
\end{prop}

If $\Ical\subset\Scal$, we denote by $\Rcal_\Ical^+$ the set of positive roots generated by simple roots of $\Ical$. Then we define $\rho$ to be the half sum of positive roots, and $\rho^P$ to be the half sum of positive roots that are not in $\Rcal_\Ical^+$ (in particular, $\rho^B=\rho$).

It is well known that an anticanonical divisor of $G/P$ is $D_{2\rho^P}$.
Anticanonical divisors of Bott-Sameslon resolutions are also well-known.

\begin{prop}(\cite[Proposition~2]{Ramanathan})\label{prop:antican}
An anticanonical divisor of $Z/P$ is $\phi^*(D_\rho)+\sum_{i=1}^NE_i$.
\end{prop}

\begin{cor}\label{cor:kltflag}
The pair $(G/P,D)$  (with $\lfloor D\rfloor=0$ as in Theorem~\ref{th:kltflag}) is klt if and only if for any $\beta$ in $\Rcal^+\backslash\Rcal^+_\Ical$,
$$
\langle 2\rho^P-\rho-\sum_{\alpha\in \Scal\backslash\Ical}d_\alpha\varpi_\alpha,\beta^\vee\rangle> 0.
$$
\end{cor}

\begin{proof}
By Propositions~\ref{prop:pullback} and \ref{prop:antican}, we get $$\begin{array}{rcl}
K_{Z/P}-\phi^*(K_{G/P}+D) & = & -\phi^*(D_\rho)-\sum_{i=1}^NE_i +\phi^*(D_{2\rho^P})-\phi^*(D)\\
 & = & \phi^*(D_{2\rho^P-\rho-\sum_{\alpha\in \Scal\backslash\Ical}d_\alpha\varpi_\alpha})-\sum_{i=1}^NE_i\\
  & = & \sum_{i=1}^N(\langle 2\rho^P-\rho-\sum_{\alpha\in \Scal\backslash\Ical}d_\alpha\varpi_\alpha,\beta^\vee_i\rangle -1)E_i.
\end{array}$$

We conclude by remarking that, since $s_{\alpha_1}s_{\alpha_2}\cdots s_{\alpha_N}$ is a reduced expression of $w_0^P$, the set $\{\beta_i\,\mid\,i=1\cdots N\}$ is $\Rcal^+\backslash\Rcal^+_\Ical$.
\end{proof}

The condition of Corollary~\ref{cor:kltflag} is always satisfied by Proposition~\ref{prop:calculrootsys} and the hypothesis that $\lfloor D\rfloor=0$. Then Theorem~\ref{th:kltflag} is proved.

\section{Horospherical pairs}

From the classification of horospherical $G$-varieties, the description of $G$-equivariant morphisms between horospherical $G$-varieties, the description of $B$-stable Cartier divisor of horospherical $G$-varieties and the  description of a $B$-stable anticanonical divisor of horospherical $G$-varieties (see for example \cite{Fanohoro}), we have the following result.

\begin{prop}\label{prop:toroidal}
Let $X$ be a horospherical $G$-variety with open $G$-orbit isomorphic to $G/H$, torus fibration over the flag variety $G/P$. Then, there exists a smooth toric $P/H$-variety $Y$ and a $G$-equivariant birational morphism $f$ from the smooth horospherical $G$-variety $V:=G\times^PY$ to $X$, such that the exceptional locus of $f$ is of pure codimension one.

Let $D$ be a $B$-stable effective $\Qbb$-divisor of $X$ such that $\lfloor D\rfloor=0$.
Write $K_V-f^*(K_X+D)=-f_*^{-1}(D)+\sum_{i\in \mathcal{E}}a_iV_i$. Then, for any $i\in\mathcal{E}$,  $V_i$ is exceptional and $G$-stable, in particular there exists a $P$-stable divisor $Y_i$ of $Y$ such that $V_i=G\times^PY_i$. Moreover, $a_i>-1$ for any $i\in \mathcal{E}$.
\end{prop}

We do not want here to recall the long description and theory of horospherical varieties. To get more details, see for example \cite{Fanohoro} or \cite{SingSpher}.

\begin{proof}
With the description in terms of colored fans of horospherical $G$-varieties and $G$-equivariant morphisms between them, $Y$ can be chosen as the toric $P/H$-variety associated to a smooth subdivision $\mathbb{F}_Y$ of the fan associated to the colored fan $\mathbb{F}_X$ of $X$. Then we clearly have that $V:=G\times^PY$ is smooth and associated to the fan $\mathbb{F}_Y$ considered as a colored fan without color. In particular, there exists a $G$-equivariant morphism from $V:=G\times^PY$ to $X$. 

Moreover, we can choose $\mathbb{F}_Y$ such that:
\begin{itemize}
\item each image of a color of $\mathbb{F}_X$ is in an edge of $\mathbb{F}_Y$ and,
\item each cone of $\mathbb{F}_Y$ that is not a cone of $\mathbb{F}_X$ contains an edge that is in $\mathbb{F}_Y$ but not in $\mathbb{F}_X$.
\end{itemize}

These two conditions implies that the exceptional locus of $f$ is of pure codimension one.\\

Any exceptional divisor $V_i$ of $f$ is $G$-stable and of the form $G\times^PY_i$ where $Y_i$ is a $P$-stable divisor of $Y$.

It remains to prove that $a_i>-1$ for any $i\in \mathcal{E}$. We use that $-K_X=\sum_{i=1}^mX_i+\sum_{\alpha\in\Scal\backslash\Ical}a_\alpha D_\alpha$ where the $X_i$'s are the $G$-stable irreducible divisors of $X$ and the $a_\alpha$ are positive integers.
 Similarly, with our notation, $K_V=-\sum_{i=1}^{m}X_i-\sum_{i\in\mathcal{E}}f_*^{-1}(X_i)-\sum_{\alpha\in\Scal\backslash\Ical}a_\alpha D_\alpha$.
In particular, by hypothesis on $D$, we remark that the divisor $-K_X-D$ is strictly effective (ie,$\sum_{i=1}^mb_iX_i+\sum_{\alpha\in\Scal\backslash\Ical}b_\alpha D_\alpha$, with $b_i>0$ for any $i\in\{1,\dots,m\}$ and $b_\alpha>0$ for any $\alpha\in\Scal\backslash\Ical$) and then, by the description of pull-backs of $B$-stable divisors of horospherical varieties, $f^*(-K_X-D)$ is also strictly effective. Hence, we have $a_i>-1$ for any $i\in \mathcal{E}$.
\end{proof}

\begin{teo}
Let $X$ be a horospherical $G$-variety. 
Let $D$ be any $B$-stable $\Qbb$-divisor $D$ of $X$ such that $\lfloor D\rfloor=0$, then $(X,D)$ has klt singularities.
\end{teo}

\begin{proof}

Let $f$ be as in Proposition~\ref{prop:toroidal} and let $Z$ be as in Theorem~\ref{th:kltflag}. Define $V':=Z\times^PY$ and let $\pi:\,V'\longrightarrow V$ the natural $B$-equivariant morphism defined from $\phi$. \\

We first prove that the $B$-equivariant morphism $f\circ \pi:\,V'\longrightarrow X$ is a log resolution of $(X,D)$. By composition, it is clearly a birational morphism and its exceptional locus is the union of the inverse images $Z\times^PY_i$ of the exceptional divisors of $f$ and the exceptional divisors $F_i\times^PY$ of $\pi$ (the exceptional locus of $\pi$ is of pure codimension one because $V$ is smooth).

The divisor $(f\circ \pi)_*^{-1}(D)+\sum_{E\in\rm{Exc}(f\circ \pi)}E$ is a $B$-stable divisor of $V'$ and then has simple normal crossings. Indeed, a $B$-stable irreducible divisor of $V'$ is either $F_i\times^PY$ where $F_i$ is one of the $B$-stable irreducible divisors of $Z$ described in Proposition~\ref{prop:divBS}, or $Z\times^PY_i$ where $Y_i$ is a $P$-stable divisor of $Y$. (Recall that, any divisor of a smooth toric variety that is stable under the action of the torus has simple normal crossings, because such a variety is everywhere locally isomorphic to $\Cbb^n$ with the natural action of $(\Cbb^*)^n$.)\\


Since $D$ is $B$-stable, we have $D=\sum_{i=1}^md_iX_i+\sum_{\alpha\in\Scal\backslash\Ical}d_\alpha D_\alpha$ where the $X_i$'s are the $G$-stable irreducible divisors of $X$. We denote by $D_B$ the $B$-stable but not $G$-stable part $\sum_{\alpha\in\Scal\backslash\Ical}d_\alpha D_\alpha$ of $D$. Then we decompose $K_{V'}-(f\circ\pi)^*(K_X+D)$ as follows: $$(K_{V'}-\pi^*(K_V+f_*^{-1}(D_B)))+\pi^*(K_V-f^*(K_X+D)+f_*^{-1}(D_B)).$$

By Proposition~\ref{prop:toroidal}, $K_V-f^*(K_X+D)+f_*^{-1}(D_B)=\sum_{i\in\mathcal{E}}a_iV_i$, where for any $i\in \mathcal{E}$, $a_i>-1$ and $V_i=G\times^PY_i$ with some $P$-stable irreducible divisor $Y_i$ of $Y$. We remark that the inverse image of $V_i$ by $\pi$ is the irreducible divisor $Z\times^PV_i$ so that $\pi^*(V_i)=Z\times^PY_i$.
Hence, $\pi^*(K_V-f^*(K_X+D)+f_*^{-1}(D))=\sum_{i\in\mathcal{E}}a_iZ\times^PY_i.$\\

To compute $K_{V'}-\pi^*(K_V+f_*^{-1}(D_B))$, we use the fibrations $p:\,V=G\times^PY\longrightarrow G/P$ and $p':\,V'=Z\times^PY\longrightarrow Z/P$, which have the same fiber. To summarize, we get the following commutative diagram.

$$ \xymatrix{
V'=Z\times^PY\ar[r]^{\pi} \ar[d]^{p'} & V=G\times^PY \ar[r]^-{f} \ar@{->}[d]^{p} &  X\\
Z/P\ar[r]^{\phi} & G/P & \\}
$$

 In particular, we have $K_V=p^*(K_{G/P})+K_p$ and $K_{V'}=p^*(K_{Z/P})+K_{p'}$. Moreover, the relative canonical divisors $K_{p'}$ and $K_p$ satisfy $K_{p'}=\pi^*(K_p)$.
 
Moreover, for any $B$-stable irreducible divisor $D$ of $V$ that is not $G$-stable, $D$ is the pull-back by $p$ of a Schubert divisor of $G/P$, in particular $D=p^*(p_*(D))$.

Hence, we get \begin{align*}
K_{V'}-\pi^*(K_V+f_*^{-1}(D_B)) & =  p'^*(K_{Z/P})+K_{p'}-\pi^*p^*(K_{G/P})-\pi^*(K_p) -\pi^*(f^{-1}_*(D_B))\\
 & = p'^*(K_{Z/P})+\pi^*(K_p)-p'^*\phi^*(K_{G/P})-\pi^*(K_p)\\ 
 & \hspace{6cm}
 -\pi^*(p^*p_*(f^{-1}_*(D_B)))\\
  & = p'^*(K_{Z/P}-\phi^*(K_{G/P}+p_*(f^{-1}_*(D_B))).
  \end{align*}

Remark that $\lfloor p_*(f^{-1}_*(D_B)\rfloor=\lfloor D_B\rfloor$, so that by Theorem~\ref{th:kltflag}, we get $K_{Z/P}-\phi^*(K_{G/P}+p_*(f^{-1}_*(D_B))=\sum_{i\in\mathcal{E}'}a_iF_i/P$, where for any $i\in\mathcal{E}'$, we have $a_i>-1$ and $F_i$ is a $B\times P$-stable irreducible divisor of $Z$.

Hence, we have $K_{V'}-\pi^*(K_V+f_*^{-1}(D_B))=\sum_{i\in\mathcal{E}'}a_iF_i\times^P Y$.

And finally, we have $$K_{V'}-(f\circ\pi)^*(K_X+D)=\sum_{i\in\mathcal{E}'}a_iF_i\times^P Y+\sum_{i\in\mathcal{E}}a_iZ\times^PY_i,$$ with, for any $i\in\mathcal{E}'\cup\mathcal{E}$, $a_i>-1$.
\end{proof}

\section{A result on root systems}

In that independent section, we prove the result that permits to deduce Theorem~\ref{th:kltflag} from Corollary~\ref{cor:kltflag}. We keep notations of section~2 and we recall that, if $\Ical\subset\Scal$, we denote by $\Rcal_\Ical^+$ the set of positive roots generated by simple roots of $\Ical$, $\rho$ denotes the half sum of positive roots, and $\rho^P$ denotes the half sum of positive roots that are not in $\Rcal_\Ical^+$.

\begin{prop}\label{prop:calculrootsys}
For any (proper) parabolic subgroup $P$ of $G$ containing $B$, and for any $\beta$ in $\Rcal^+\backslash\Rcal^+_\Ical$,
\begin{equation}\label{eq1}
\langle 2\rho^P-\rho-\sum_{\alpha\in \Scal\backslash\Ical}\varpi_\alpha,\beta^\vee\rangle\geq 0.
\end{equation}
\end{prop}

Note that $\rho=\sum_{\alpha\in \Scal}\varpi_\alpha$ and that $2\rho^P=2\rho -\sum_{\gamma\in\Rcal^+_\Ical}\gamma=2\sum_{\alpha\in \Scal}\varpi_\alpha-\sum_{\gamma\in\Rcal^+_\Ical}\gamma$. Hence, equation~\ref{eq1} is equivalent to
 
\begin{equation}\label{eq2}
\langle\sum_{\alpha\in \Ical}\varpi_\alpha-\sum_{\gamma\in\Rcal^+_\Ical}\gamma,\beta^\vee\rangle\geq 0.
\end{equation}


\begin{rem}
\begin{enumerate}[label=(\roman*)]
\item If $\Ical=\emptyset$ (ie if $P=B$), equations~\ref{eq1} and \ref{eq2} are trivially satisfied.
\item If $\beta\in\Rcal^+_\Ical$ then $\langle\sum_{\alpha\in \Ical}\varpi_\alpha-\sum_{\gamma\in\Rcal^+_\Ical}\gamma,\beta^\vee\rangle=-\langle\sum_{\alpha\in \Ical}\varpi_\alpha,\beta^\vee\rangle$ and is negative.
\end{enumerate}
\end{rem}

\begin{lem}\label{lem:calculrootsys}
Denote by $w_{0,P}$ the longest element of $W_P$. Then, we have $$w_{0,P}(\sum_{\alpha\in \Ical}\varpi_\alpha)=\sum_{\alpha\in \Ical}\varpi_\alpha-\sum_{\gamma\in\Rcal^+_\Ical}\gamma.$$
\end{lem}

\begin{proof}
First note that, for any character $\lambda$ of $T$ and for any $w\in W_P$, $w(\lambda)-\lambda$ is in the lattice $\bigoplus_{\alpha\in\Ical}\Zbb\alpha$ (it can be easily proved by induction on the length of $w$).

Moreover, $\rho_P:=\sum_{\gamma\in\Rcal^+_\Ical}\gamma$ satisfies $\langle\rho_P,\alpha^\vee\rangle=2$ for any $\alpha\in\Ical$ (same result as the well-known result: $\langle\rho,\alpha^\vee\rangle=2$ for any $\alpha\in\Scal$). And, since $w_{0,P}$ is the longest element of $W_P$, we that for any $\alpha\in\Ical$, the root $w_{0,P}(\alpha)$ is the opposite of a simple root in $\Ical$.

Hence, if we denote by $\lambda$ the character $w_{0,P}(\sum_{\alpha\in \Ical}\varpi_\alpha)-\sum_{\alpha\in \Ical}\varpi_\alpha+\sum_{\gamma\in\Rcal^+_\Ical}\gamma$, we get
$\lambda\in \bigoplus_{\alpha\in\Ical}\Zbb\alpha$ and, for any $\alpha\in\Ical$, we have $$\langle\lambda,\alpha^\vee\rangle=\langle\sum_{\alpha\in \Ical}\varpi_\alpha,w_{0,P}(\alpha^\vee)\rangle-\langle\sum_{\alpha\in \Ical}\varpi_\alpha,\alpha^\vee\rangle+\langle\rho_
P,\alpha^\vee\rangle=-1-1+2=0.$$ Then, we deduce that $\lambda=0$, which proves the lemma.
\end{proof}

\begin{proof}[Proof of Proposition~\ref{prop:calculrootsys}]
By Lemma~\ref{lem:calculrootsys}, we get for any $\beta\in\Rcal^+\backslash\Rcal_\Ical^+$,
$$\langle\sum_{\alpha\in \Ical}\varpi_\alpha-\sum_{\gamma\in\Rcal^+_\Ical}\gamma,\beta^\vee\rangle=\langle\sum_{\alpha\in \Ical}\varpi_\alpha,w_{0,P}(\beta^\vee)\rangle.$$

But, the positive roots that are send to negative roots by $w_{0,P}$ are exactly the roots in $\Rcal_\Ical^+$. In particular, for any $\beta\in\Rcal^+\backslash\Rcal_\Ical^+$, $w_{0,P}(\beta^\vee)$ is a positive coroot, and $$\langle\sum_{\alpha\in \Ical}\varpi_\alpha,w_{0,P}(\beta^\vee)\rangle\geq 0.$$
\end{proof}

\begin{rem}
Let $\beta\in\Rcal^+\backslash\Rcal_\Ical^+$. Then
$\langle\sum_{\alpha\in \Ical}\varpi_\alpha-\sum_{\gamma\in\Rcal^+_\Ical}\gamma,\beta^\vee\rangle=0$ if and only if $w_{0,P}(\beta)\in\Rcal_{\Scal\backslash\Ical}^+$. In particular, if $\Ical\neq\Scal$, there exists $\beta\in\Rcal^+\backslash\Rcal_\Ical^+$ such that $\langle\sum_{\alpha\in \Ical}\varpi_\alpha-\sum_{\gamma\in\Rcal^+_\Ical}\gamma,\beta^\vee\rangle=0$.
\end{rem}
\date{\textit{Acknowledgement: }I would like to thank C\'edric Bonnaf\'e who points me how to give a shorter proof of Proposition~\ref{prop:calculrootsys}.

\bibliographystyle{amsalpha}
\bibliography{SingSpher}
\bigskip\noindent

\end{document}